\documentclass[12pt]{amsart} 
\usepackage{amscd}
\usepackage{amssymb}
\usepackage{a4wide}
\usepackage{amstext}
\usepackage{amsthm}
\usepackage{xcolor}
\usepackage{cite}
\usepackage[T1,T2A]{fontenc}
\usepackage[utf8]{inputenc}
\usepackage[american]{babel}
\usepackage{url}
\usepackage{amsfonts}
\usepackage{amssymb, amsthm}
\usepackage{amsmath}
\usepackage{mathtools}
\usepackage{needspace}
\usepackage[pdftex]{graphicx}
\usepackage{hyperref}
\usepackage{datetime}
\usepackage{epigraph}
\usepackage{verbatim}
\usepackage{mathtools}
\usepackage{xcolor}
\usepackage{enumerate}
\usepackage[american]{babel}
\numberwithin{equation}{section}

\newcommand{\Z}{\mathbb{Z}}

\newcommand{\N}{\mathbb{N}}
\newcommand{\R}{\mathbb{R}}

\newcommand{\Cm}{\mathbb{C}}

\newcommand{\summ}{\sum\limits}

\newcommand{\eps}{\varepsilon}

\newcommand{\Hg}{\mathcal{H}}

\DeclareMathOperator{\supp}{supp} 

\newcommand{\F}{\mathcal{F}}
\newcommand{\pho}{\rho}
\renewcommand{\phi}{\varphi}

\newtheorem{Thm}{Theorem}[section]
\newtheorem{theorem}[Thm]{Theorem}

\newtheorem{corollary}[Thm]{Corollary}
\newtheorem{remark}[Thm]{Remark}

\begin{document}
\sloppy
\title[Fourier interpolation and time-frequency localization]
{Fourier interpolation and time-frequency localization}
\author{Aleksei Kulikov}
\address{Department of Mathematical Sciences, Norwegian University of Science and Technology, NO-7491 Trondheim, Norway
\newline {\tt lyosha.kulikov@mail.ru}
}
\begin{abstract} { We prove that under very mild conditions for any interpolation formula $f(x) = \sum_{\lambda\in \Lambda} f(\lambda)a_\lambda(x) + \sum_{\mu\in M} \hat{f}(\mu)b_{\mu}(x)$ we have a lower bound for the counting functions $n_\Lambda(R_1) + n_{M}(R_2) \ge 4R_1R_2 - C\log^{2+\eps}(4R_1R_2)$ which very closely matches interpolation formulas from \cite{MR3949027}, \cite{zeta}.
}
\end{abstract}
\maketitle
\section{Introduction}
In the recent breakthrough paper \cite{MR3949027}  Radchenko and Viazovska showed that any Schwartz function can be effectively reconstructed from the values of it and its Fourier transform at the points $\pm \sqrt{n}, n\in \Z_{\ge 0}$ and two more values $f'(0)$, $\hat{f}'(0)$.  If we consider the counting function $n_\Lambda(R) = |\Lambda \cap [-R, R]|$, which in the case $\Lambda = \{ \pm \sqrt{n}\}$ takes the form $n_\Lambda(R) = 1 + 2[R^2]$, we  see that it satisfies the inequality $n_\Lambda(W) + n_\Lambda(T) \ge 4WT - O(1)$ for all $W, T$. 
 We  observe that this bound perfectly matches the famous $4WT$ Theorem of Slepian \cite{MR0462765} which says that the space of functions which are supported on $[-T, T]$ and such that their Fourier transforms are essentially supported on $[-W, W]$ has  approximate dimension $4WT$.\footnote{Slepian called it the $2WT$ Theorem since he considered intervals $[-W, W]$ and $[-T/2, T/2]$ which is more natural from the engineering point of view.} 
We prove that this is not a coincidence and that a similar inequality holds for all such interpolation formulas with a very small error term.
\begin{theorem}\label{Theorem}\label{main}
Let $\Lambda, M\subset \R$ be two  multisets and $L$ be some fixed number. Assume that the following interpolation formula holds for all $C^\infty$ compactly supported functions $f$ 
\begin{equation}\label{formula}
f(x) = \summ_{\lambda\in \Lambda} f^{(r(\lambda))}(\lambda)a_\lambda(x) + \summ_{\mu\in M}\hat{f}^{(k(\mu))}(\mu)b_\mu(x),
\end{equation}
where $r:\Lambda \to \Z_{\ge 0}, k:M\to \Z_{\ge 0}$ and $a_\lambda, b_\mu :\R \to \Cm, \lambda \in \Lambda, \mu\in M$. Assume additionally that $k$ is at most $L$,  the counting function of $M$ satisfies the bound $n_M(R) \le R^L$ for large enough $R$ and that $b_\mu(x)$ is polynomially bounded in $\mu$ and $x$. Then for any $\eps > 0$ there exists $C > 0$ (depending on $\eps$ and the interpolation formula) such that for all $R_1, R_2 > 1$ 
\begin{equation}\label{final bound}
n_\Lambda(R_1) + n_M(R_2) \ge 4R_1R_2 - C\log^{2+\eps}(4R_1R_2).
\end{equation}
\end{theorem}


This result reflects the idea that for a function from  Slepian's theorem values of this function and its Fourier transform outside of the corresponding intervals are mostly irrelevant and to generate an $N$-dimensional vector space we need at least $N$ vectors (although the ``space'' under consideration is by no means a vector space).

The way to make the $4WT$ Theorem of Slepian precise is to consider the so-called prolate spheroidal wave functions, studied by Slepian, Landau and Pollak \cite{MR140732}, \cite{MR179392}, \cite{MR147686}. These functions are the eigenvectors of the time-frequency localization operator corresponding to the intervals $[-W, W]$ and $[-T, T]$. The key ingredient in our proof is the sharp  estimate for the corresponding eigenvalues obtained by Israel \cite{israel}. However, while we can prove Theorem \ref{main} using the prolate spheroidal wave functions, it is much more natural for us to use directly the basis functions constructed in \cite{israel} because of the decay of their Fourier transforms.

Let us emphasize that for this theorem it is not enough to assume only  uniqueness  i.e. that any function which vanishes on $\Lambda$  and whose Fourier transform vanishes on $M$ is zero. We need some quantitative assumption like an  interpolation formula or frame property. To illustrate this, let us mention a well-known result of Ascensi, Lyubarskii and Seip \cite{MR2490219}, which gives us a uniqueness result (but not an interpolation formula) with effectively half the number of points.

\begin{theorem}[Ascensi, Lyubarskii, Seip]\label{gauss}
Let $f\in L^2(\R)$. Assume that it is orthogonal to the functions $\exp(-\pi (t+\lambda)^2), \lambda \in \Lambda$ and $\exp(-\pi t^2 + 2\pi i\mu t), \mu \in M$ with $\Lambda = \{ \pm \sqrt{2n}\} $ and $M = \{ \pm \sqrt{2n}\}\cup \{ -1, 1\}$.  Then this function is identically zero.
\end{theorem}

If we think of a gaussian as a smoothed version of the $\delta$-function used in formula \eqref{formula} then this result gives us a uniqueness set with half as many points as claimed in the Theorem \ref{main}. Nevertheless there is no contradiction between Theorem \ref{main} and Theorem ~\ref{gauss} because there is no efficient way to reconstruct a function $f$ from its scalar products with the functions from the Theorem \ref{gauss}. 

The usual way to express effective reconstruction is by imposing the frame property. Recall that the set of vectors $\{v_k\}$ in the Hilbert space $\Hg$ is said to be a frame if for all $f\in \Hg$ we have $||f||^2 \sim \sum_k |\langle f, v_k\rangle |^2$. In this language we can say that the set of functions from  Theorem \ref{gauss} is extremely far from being a frame (even if we only care about a lower bound for $||f||^2$ and put a fairly large weight on $|\langle f, v_k\rangle |^2$), which can be deduced from the (proof of the) general result of Seip \cite{MR1173117} or seen directly by considering the functions $f_N(t) = \exp(-\pi (t + N)^2 + 2\pi i Nt), N\to \infty$.  Theorem \ref{main} shows that a density condition akin to that of \cite{MR1173117} holds under a much weaker assumption about reconstruction than what follows from the frame property.

Let us also mention  a recent interpolation formula of Bondarenko, Radchenko and Seip \cite{zeta}. They proved that, under suitable conditions, one can recover the value $\hat{f}(x)$ by means of an interpolation formula from the values of $f$ at the points $\frac{\pm\log (n)}{4\pi}$ and the values of $\hat{f}$ at the points $\frac{\rho - 1/2}{i}$ with $\rho$ ranging over the nontrivial zeros of the Riemann zeta function. Although in the absence of the Riemann hypothesis these points can be non-real and the formula from \cite{zeta} converges only after some grouping of terms, one can still apply our techniques to their setting and get the bound
\begin{equation}
2N(T) + 2e^{4\pi W} \ge 4WT - C\log^{2+\eps}(WT),
\end{equation}
where $N(T)$ is the number of  zeros $\pho$ of the Riemann zeta function with $0 < \Im (\pho) < T$. Choosing $W = \frac{1}{4\pi}\log (\frac{T}{2\pi})$ we get the lower bound which matches the Riemann-von Mangoldt formula up to the power of the logarithmic term
\begin{equation}
N(T) \ge \frac{T}{2\pi}\log\left(\frac{T}{2\pi e}\right) - C\log^{2+\eps}(T).
\end{equation}

Finally let us remark that our result admits a natural generalization to the space of even/odd functions with $4WT$ replaced by $2WT$. This result also perfectly matches interpolation formulas from \cite{MR3949027} and \cite{zeta}.


\section{Local cosine basis}
A local cosine basis is an orthonormal basis of $L^2(I)$ associated with the Whitney decomposition of $I$ that was introduced by Coifman and Meyer in \cite{MR1089710} as a tool for smooth localization in Fourier analysis (see also \cite{MR1161254} for a nice exposition of this and related matters). In \cite{israel} Israel constructed local cosine basis functions which in addition have a strong decay of their Fourier transforms. We shall now list  properties of this basis that we will use in our proof.

Let $0 < \eta < 1$ be a fixed real number. For $D\ge 2$ let $\{ I_j\}_{j\in J}$ be the Whitney decomposition of the interval $I = [-\frac{D}{2}, \frac{D}{2}]$ and set $\delta_j = |I_j|$. There exists an orthonormal basis $\{ \Phi_{j, k}\}_{j\in J, k\in \Z_{\ge 0}}$ of $L^2(I)$ with the following properties: $\Phi_{j, k}$ are $C^{\infty}(\R)$, compactly supported functions and they satisfy the following Fourier concentration inequality (inequality $(23)$ from \cite{israel})
\begin{equation}\label{concentration}
|\F (\Phi_{j, k})(\xi)| \le \delta_j^{1/2} A_\eta \left(\exp(-a_\eta (\delta_j |\xi - \xi_{j, k}|)^{1-\eta}) + \exp(-a_\eta (\delta_j |\xi + \xi_{j, k}|)^{1-\eta})\right),
\end{equation}
where $\xi_{j,k} = \frac{2k+1}{4\delta_j}$ and $A_\eta, a_\eta$ are some positive constants.

For our purposes we need a similar bound for the derivatives of $\F(\Phi_{j, k})$. To get them we will use the following real-analytic inequality.
\begin{theorem}[Landau-Kolmogorov inequality on the half-line]
For any $n, k\in \N$, $k < n$ there exists a constant $C_{n, k} > 0$ such that for all $f\in C^{\infty}([0, +\infty))$ we have
\begin{equation}\label{Kolmogorov}
||f^{(k)}||_{L^\infty (0, +\infty)} \le C_{n, k} ||f||^{1-k/n}_{L^\infty (0, +\infty)} ||f^{(n)}||^{k/n}_{L^\infty (0, +\infty)}.
\end{equation}
\end{theorem}
For the proof and various generalizations of this inequality see e.g. \cite{Kolm}.

Applying this theorem to the functions $\Phi_{j, k}$ we get the following corollary. 

\begin{corollary}
For any $n\in \N$ there exist positive constants $C_n, d_n, s_n$ (possibly also depending on $\eta$) such that for $|\xi| > \xi_{j, k}$ we have
\begin{equation}\label{sharp derivative}
|\F (\Phi_{j, k})^{(n)}(\xi)| \le C_nD^{d_n}\left(\exp(-s_n(\delta_j |\xi - \xi_{j, k}|)^{1-\eta}) + \exp(-s_n(\delta_j |\xi + \xi_{j, k}|)^{1-\eta})\right).
\end{equation}
\end{corollary}
\begin{proof}
Without loss of generality we may assume that $\xi > \xi_{j, k}$ (the other case is completely analogous). Consider the function $f(x) = \F(\Phi_{j, k})(x + \xi)$. By \eqref{concentration} we have for $x\ge 0$
\begin{equation}
|f(x)| \le 2D^{1/2}A_\eta \exp(-a_\eta (\delta_j |x + \xi- \xi_{j, k}|)^{1-\eta}) \le 2D^{1/2}A_\eta \exp(-a_\eta (\delta_j | \xi- \xi_{j, k}|)^{1-\eta}) ,
\end{equation}
where we used $\delta_j \le D$ and $|\xi - \xi_{j, k}| \le|x + \xi - \xi_{j, k}| \le |x + \xi + \xi_{j, k}|$. 

On the other hand we have $\F(\Phi_{j, k})^{(m)} = \F ((2\pi i x)^m \Phi_{j, k})$. Therefore by the Cauchy--Schwarz inequality we have for all $t \in \R$
\begin{equation}
|\F(\Phi_{j, k})^{(m)}(t)| \le C_mD^{m + 1/2}.
\end{equation}
Since $f$ is just a shift of $\F(\Phi_{j, k})$ the same bound holds for $f$ as well.

Choosing $m = 2n$ and applying \eqref{Kolmogorov} to $f$ we get
\begin{equation}
|f^{(n)}(x)| \le C_{2n, n}D^{n + 1/2} \sqrt{2A_\eta C_{2n}}\exp(-\frac{a_\eta}{2}(\delta_j | \xi- \xi_{j, k}|)^{1-\eta}).
\end{equation}
Applying this for $x = 0$ and recalling that $f(x) = \F (\Phi_{j, k})(x+\xi)$ we get the desired inequality.
\end{proof}
\begin{remark}\label{rem}
Using the methods from \cite{israel} directly instead of applying the general bound \eqref{Kolmogorov}, we can get much better quantitative bounds for the constants $C_n, d_n, s_n$. But since it is irrelevant for our applications, we decided to follow this route to simplify the exposition.
\end{remark}
Let us deduce another corollary from the bound \eqref{sharp derivative} which is more convenient for our applications. 
\begin{corollary}
For all $T_1, T_2, n \ge 0$ there exist constants $C, c >0$ such that for $k < \delta_j - C\log^{1/(1-\eta)}(D)$ and $|\xi| > \frac{1}{2}$ we have
\begin{equation}\label{cor}
|\F(\Phi_{j, k})^{(n)}(\xi)| \le \frac{c}{D^{T_1}|\xi|^{T_2}}.
\end{equation}
\end{corollary}
\begin{proof}
We consider two cases: $|\xi| <D$ and $|\xi| \ge D$. In the first case we have by \eqref{sharp derivative} and our assumption that $|\xi|> \frac{1}{2}$
\begin{equation}
|\F (\Phi_{j, k})^{(n)}(\xi)| \le 2C_nD^{d_n}\left(\exp(-s_n(\delta_j (\frac{1}{2} - \xi_{j, k})^{1-\eta})\right).
\end{equation}

Since $\xi_{j, k} = \frac{2k+1}{4\delta_j}$ and $k <  \delta_j - C\log^{1/(1-\eta)}(D)$ we have $\delta_j(\frac{1}{2} - \xi_{j, k}) > \frac{C}{2}\log^{1/(1-\eta)}(D) - \frac{1}{4} \ge \frac{C}{4}\log^{1/(1-\eta)}(D)$ for $C \ge 10$. Therefore we get
\begin{equation}
|\F (\Phi_{j, k})^{(n)}(\xi)| \le 2C_nD^{d_n-s_n\left(\frac{C}{4}\right)^{1/(1-\eta)}}.
\end{equation}
Choosing $C$ so that $d_n-s_n\left(\frac{C}{4}\right)^{1/(1-\eta)} < -(T_1+T_2)$ we get the desired result.

In the second case  \eqref{sharp derivative} gives us
\begin{equation}
|\F (\Phi_{j, k})^{(n)}(\xi)|  \le 2C_nD^{d_n} \exp(-s_n\delta_j(|\xi| - \xi_{j, k})^{1-\eta}).
\end{equation}
If $\delta_j < 1$ then $k < 1 - C\log^{1/(1-\eta)}(D) < 0$ and the claim is vacuously true. Therefore it is enough to consider the case $\delta_j \ge 1$. Since $D \le |\xi|$, $\xi_{j, k} \le \frac{1}{2}$ it suffices to prove that for $|\xi| \ge \frac{1}{2}$ we have
\begin{equation}
2C_n|\xi|^{d_n} \exp\left(-s_n \left(|\xi| - \frac{1}{2}\right)^{1-\eta}\right) \le \frac{c}{|\xi|^{T_1+T_2}},
\end{equation}
which is true for some constant $c$ since $\exp(t^{1-\eta})$ grows faster than any polynomial.
\end{proof}

For our purposes we will choose $\eta = 1 - \frac{1}{1+\eps}$ so that $\frac{1}{1-\eta} = 1 + \eps$. Put $J' = \{ j\mid \delta_j \ge 1\}$. Since $\{ I_j\}_{j\in J}$ is a Whitney decomposition of $I$ we have  $\sum_{j\in J'} \delta_j \ge |I| - O(1) = D-O(1)$ and there are $O(\log (|I|)) = O(\log (D))$ elements in $J'$. In the proof of Theorem \ref{main} we will work with the set $S = \{ (j, k) \mid j\in J', k < \delta_j - C\log^{1+\eps}(D)\}$. Its size is at least $\sum_{j\in J'} \delta_j - C|J'| \log^{1+\eps}(D) = D - C\log^{2+\eps}(D)$.

\section{proof of Theorem \ref{main}}
Put $D = 4R_1R_2$ and assume that $n_{\Lambda}(R_1) + n_M(R_2) < D - C\log^{2+\eps}(D)$. Consider the function $f(x) = \sum_{(j, k)\in S} a_{j, k}\Phi_{j, k}(2R_2x)$. By a linear algebra argument and the lower bound for $|S|$ we can find $a_{j, k}$ such that $f^{(r(\lambda))}(\lambda) = 0, |\lambda| \le R_1$, $\hat{f}^{(k(\mu))}(\mu) = 0, |\mu| \le R_2$ and $\sum |a_{j, k}|^2 = 1$.

Since $\Phi_{j, k}$ are supported on $[-\frac{D}{2}, \frac{D}{2}]$, the function $f$ is supported on $[-R_1, R_1]$. Since the functions $\Phi_{j, k}$ are orthonormal in $L^2(-\frac{D}{2}, \frac{D}{2})$ we have that $||f||_{L^2(\R)} = \frac{1}{\sqrt{2R_2}}$ and therefore there exists $x\in [-R_1, R_1]$ such that $|f(x)| \ge \frac{1}{\sqrt{D}}$. Consider  formula \eqref{formula} with this $x$. By construction we have $f^{(r(\lambda))}(\lambda) = 0$ for $|\lambda| \le R_1$ but since $\supp f\subset [-R_1, R_2]$ the same holds for $\lambda$ with $|\lambda| > R_1$ as well. Therefore the first half of the interpolation formula is $0$. Similarly in the second half only terms with $|\mu| > R_2$ remain. Thus we have
\begin{equation}\label{short formula}
f(x) = \summ_{\mu \in M, |\mu| > R_2} \hat{f}^{(k(\mu))}(\mu) b_\mu(x).
\end{equation}
Note that we trivially have $|a_{j, k}| \le 1$. Expanding \eqref{short formula} into $\F(\Phi_{j, k})$ and using the bound \eqref{cor} (which we can apply since after our scaling by $2R_2$, the bounds for $|x| > \frac{1}{2}$ correspond to bounds for $|x| > R_2$)
\begin{equation}
\frac{1}{\sqrt{D}} \le \summ_{\mu \in M, |\mu| > R_2} |S| (2R_2)^{-k(\mu)} \frac{c(2R_2)^{T_2}}{D^{T_1}|\mu|^{T_2}}b_\mu(x).
\end{equation}
Since $2R_2 \ge 1$  we have $(2R_2)^{-k(\mu)} \le 1$. We also have $|S| \le D$. By the assumption of the theorem we have $|b_\mu(x)| \le P(x, \mu)$ for some polynomial $P$. Since $|x| \le R_1 \le D$ and $1\le R_2 \le \mu$ we have $P(x, \mu) \le c'D^U|\mu|^U$ for some $U$. Finally since $2R_2 \le D$ we have $(2R_2)^{T_2} \le D^{T_2}$. Collecting everything  we get
\begin{equation}
\frac{1}{\sqrt{D}} \le \frac{cc'}{D^{T_1 - U -T_2- 1}}\summ_{\mu \in M, |\mu| > R_2} |\mu|^{U-T_2}.
\end{equation}

Since $n_M(R)\le R^L$ we can choose $T_2$ large enough so that the last sum converges and is bounded by some absolute constant $c''$.  Choosing $T_1 \ge U +T_2  +2$ we get
\begin{equation}
\frac{1}{\sqrt{D}} \le \frac{cc'c''}{D},
\end{equation}
which is false for large enough $D$.  For smaller $D$ we can artificially enlarge $C$ and make the conclusion of the theorem vacuously true.

If  formula \eqref{formula} is true only for even/odd functions then we put $D = 2R_1R_2$ and consider $f(x) = \sum_{(j, k)\in S} a_{j, k}\Phi_{j, k}(R_2(2x-R_1))$ for $x > 0$ and $f(-x) = \pm f(x)$. The rest is basically the same and we get the bound 
\begin{equation}
n_\Lambda(R_1) + n_M(R_2) \ge D - C\log^{2+\eps}(D) = 2R_1R_2 - C\log^{2+\eps}(2R_1R_2).
\end{equation}
\section{Concluding remerks}
We have proved that every interpolation formula satisfies a $4WT$-type theorem in a very strong form. However, we would like to point out that our methods are very robust: the only things that we need are  a sufficient number of pairwise orthogonal functions with good time-frequency localization properties and that interpolation is done by means of linear functionals. Due to that and the strong decay of the functions $\F( \Phi_{j, k})$ which is at the limit of what is allowed by the Beurling--Malliavin theorem we can significantly weaken the assumptions of the theorem as well as prove similar results in different settings (although at an expense of a worse error term). We  list below some extensions and variations which we believe may be interesting.
\begin{enumerate}[{(1)}]
\item If we use  the bound \eqref{sharp derivative} directly we can assume much less about the interpolation formula, essentially up to $|b_\mu(x)| \lesssim \exp((|x|+|\mu|)^{1-\eps})$ and $|n_M(R)| \lesssim \exp(R^{1-\eps})$. Moreover, we can even allow a slightly growing function $k(\mu)$ (see Remark \ref{rem}).

\item We can replace point evaluations (that is convolutions with the $\delta$-function) in the interpolation formula by the convolution with some other fast decaying distribution, for example the gaussian (though in that case we need to assume something about the first half of the interpolation formula as well). This gives in particular another proof of the fact that there is no interpolation formula for the functions from  Theorem \ref{gauss}.

\item We can prove a local $4WT$ Theorem for interpolation formulas. Since time-frequency shifts are isometries on $L^2(\R)$ we can consider not only intervals $[-R_1, R_1]$ and $[-R_2, R_2]$, but arbitrary pair of intervals $I, J$. In that situation we can show that the total number of points in $I$ and $J$ is at least $|I| |J| -C\log^{2+\eta}(|I| |J|)$ if $I$ and $J$ are not extremely far from the origin in comparison with their lengths.   

\item We can allow nonreal interpolation points. Since the Fourier transform of a compactly supported function is an entire function, it makes sense to speak about its value at an arbitrary complex point.  A closer examination of the methods from \cite{israel} shows that in that case the bound for the value at the point $x+iy$ is multiplied by $e^{2\pi|y|R_1}$, that is we want imaginary parts to be not very large.  Fortunately, even in the absence of the Riemann Hypothesis, this is the case for the interpolation formula from \cite{zeta} which allows us to apply our method to the setting of \cite{zeta}.
 
Moreover, in \cite{zeta} it has been proved that any interpolation formula yields  a Dirichlet series kernel satisfying a functional equation akin to that of the Riemann zeta function and conversely any interpolation formula comes from  contour integration against such a kernel. Therefore, even if the interpolation points are real, it is natural to go to the complex plane and consider a truncated contour integral, assuming in the theorem  bounds  for the kernel instead of  bounds for the interpolation functions $b_\mu(x)$.
\end{enumerate}

Finally let us mention an interesting question which requires a more substantial modification of our methods: how to generalize Theorem \ref{main} to interpolation formulas for radial functions in higher dimensions? There are recently discovered interpolation formulas in dimensions $8, 24$ \cite{unknown1} and in all dimensions larger than or equal to $5$ \cite{unknown2} which are similar to \eqref{formula}. We expect an analogous $4WT$ Theorem to be true in this case as well but the methods from \cite{israel} by themselves are not sufficient to establish this. Indeed, one of the key ideas in \cite{israel} was that multiplication by $\cos(|x|)$ on the Fourier  transform side  corresponds to the superposition of two shifts, and there is no such result in dimension higher than $1$. We will return to this problem in a forthcoming paper.

\subsection*{Acknowledgments} I would like to thank Kristian Seip for sharing his hypothesis about the potential connection between Slepian's $4WT$ Theorem and Fourier interpolation formulas. I also would like to thank Andrii Bondarenko and Danylo Radchenko for fruitful discussions. This work was supported  by Grant 275113 of the Research Council of Norway and  in part by the Moebius Contest Foundation for Young Scientists.

\bibliographystyle{siam}

\bibliography{Fourier_interpolation_and_time-frequency_localization}

\end{document}